\documentclass[12pt,reqno]{amsart}
\usepackage[english]{babel}
\usepackage[T1]{fontenc}
\usepackage{amsthm,amsfonts,amssymb,amscd}
\usepackage[all]{xy}

\newtheorem{theorem}[subsection]{Theorem}

\newtheorem{lemma}[subsection]{Lemma}

\theoremstyle{definition}

\newtheorem{proposition-definition}[subsection]{Proposition-Definition}

\theoremstyle{remark}

\newcommand{\fatdot}{{\scriptscriptstyle \bullet}}

\newcommand{\End}{\operatorname{End}\nolimits}

\newcommand{\Gr}{\operatorname{Gr}}

\newcommand{\id}{\operatorname{id}\nolimits}
\newcommand{\im}{\operatorname{im}\nolimits}

\newcommand{\Res}{\operatorname{Res}}

\newcommand{\CC}{{\mathbb C}}
\newcommand{\HH}{{\mathbb H}}

\newcommand{\ZZ}{{\mathbb Z}}

\newcommand{\PP}{{\mathbb P}}

\newcommand{\OOO}{{\mathcal O}}

\newcommand{\JJJ}{{\mathcal J}}

\newcommand{\EEE}{{\mathcal E}}

\newcommand{\HHH}{{\mathcal H}}
    
\newcommand{\FFF}{{\mathcal F}}

\newcommand{\LLL}{{\mathcal L}}

\newcommand{\NNN}{{\mathcal N}}

\newcommand{\TTT}{{\mathcal T}}
    
\newcommand{\UUU}{{\mathcal U}}
\newcommand{\VVV}{{\mathcal V}}
\newcommand{\WWW}{{\mathcal W}}

\newcommand{\YYY}{{\mathcal Y}}
\newcommand{\ZZZ}{{\mathcal Z}}

\newcommand\alp{\alpha}

\newcommand{\into}{\hookrightarrow}

\newlength{\rrrr}
\newcommand{\isom}[1]{{\settowidth{\rrrr}{$\scriptstyle{x#1x}$}
\xrightarrow{\makebox[\rrrr]{$\scriptstyle{#1}$}}
\hspace{-0.5\rrrr }\hspace{-1.1 em}
\raisebox{- 0.5 ex}{$\sim$}\hspace{0.7\rrrr }
}}
\newcommand{\isoto}{{\lra\hspace{-1.3 em}
\raisebox{ 0.6 ex}{$\textstyle\sim$}\hspace{0.8 em}}}
\newcommand\lra{{\longrightarrow}}

\newcommand\rar{\rightarrow}

\renewcommand{\bar}[1]{\overline{#1}}

\author{D. Markushevich}

\address{D. M.: Math\'ematiques - b\^{a}t. M2, Universit\'e Lille 1,
F-59655 Villeneuve d'Ascq Cedex, France}
\email{markushe@math.univ-lille1.fr}

\subjclass{14J45,14D07,14H10,14K30,70H06}

\title{Integrable systems from \\ intermediate Jacobians of 5-folds}

\begin{document}
\begin{abstract}
Given a cubic 4-fold $Y$, we provide an easy Hodge-theoretic proof of the following result
of Iliev--Manivel: the relative intermediate
Jacobian of the universal family of
cubic 5-folds $Z$ extending $Y$
is a Lagrangian fibration.
\end{abstract}
\maketitle

\section*{Introduction}

Iliev and Manivel \cite{IM} introduced a symplectic structure on the relative intermediate Jacobian
$\JJJ=\JJJ(\VVV/B)\to B$ of the family of prime
Fano threefolds $\VVV\to B$ of index 1 extending a fixed K3
surface $S$, in such a way that $\JJJ\to B$ turns into a Lagrangian
fibration. Here, when we say that a 3-fold $V$ extends $S$, we mean that
$S$ is a hyperplane section of $V$, and $B$ is understood as a
moduli (or universal deformation) space of maps $S\into V$ with fixed source $S$, see 
\cite{B2} for the particular case of K3-Fano pairs and \cite{Ran} for deformations
of general maps.

The approach of  \cite{IM} is based upon
the observation of Tyurin that for a K3-Fano pair
$S\subset V$ with $S\in |-K_V|$, the map sending stable sheaves on $V$ to their restrictions to $S$
under certain hypotheses embeds the moduli spaces of sheaves on $V$ as 
Lagrangian subvarieties of moduli spaces on $S$. Tyurin \cite{Tyu}
originally observed this for stable vector bundles,
and Thomas \cite{Th} extended to the 
ideal sheaves of curves on $V$, whose restrictions to $S$ are ideal sheaves
of zero-dimensional subschemes. The authors of \cite{IM}
apply this to conics in $V$ and thus obtain a Lagrangian
immersion $\FFF(V)\to S^{[2]}$, where $\FFF(V)$ is the Fano surface
of $V$, parametrizing conics in $V$, and $S^{[2]}$ is the Fujiki--Beauville
symplectic 4-fold \cite{B1}, or the Hilbert scheme of 0-dimensional subschemes
of $S$ of length 2.

A 30-year long history of the study of the Abel--Jacobi map for Fano 3-folds provides a case-by-case proof of the following theorem: for a generic prime Fano 3-fold $V$ of index 1, the intermediate Jacobian $J(V)$
is canonically isomorphic to the Picard variety of $\FFF(V)$. Then $J(\VVV/B)$ is identified
with the relative Picard variety of a complete deformation family of Lagrangian
subvarieties of $S^{[2]}$, and the relative Picard variety is itself
Lagrangian by a result of Donagi--Markman \cite{DM}.

In \cite{M}, a purely Hodge theoretic proof of a more general result is given, providing the relative
intermediate Jacobian $\JJJ(\VVV/B)\to B$ with a Lagrangian fibration structure  for
all pairs $S\subset V$ with $S\in |-K_V|$, where $V$ may be a Fano 3-fold of any index and any Betti number $b_2$ (see the classification of Fano 3-folds in \cite{MM}). 

Later, Iliev--Manivel \cite{IM2} found an example of a similar situation in higher dimension: for a generic
cubic 4-fold $Y$ in $\PP^5$, consider the universal family $\ZZZ\to B$ of cubic 5-folds $Z\subset
\PP^6$ extending $Y$. Then the planes in each 5-fold $Z$ form a surface $\FFF (Z)$, and the intersections of the planes with $Y$ are lines in $Y$. According to \cite{BD}, the variety of lines $\FFF(Y)$ is a symplectic 4-fold. Iliev--Manivel prove that the intersection map $\FFF (Z)\to \FFF (Y)$ is a Lagrangian immersion. By
\cite{Col}, the intermediate Jacobian $J(Z)$ is isomorphic to the Picard variety of $\FFF (Z)$ (of dimension 21), and the Lagrangian structure of the relative intermediate Jacobian $\JJJ(\ZZZ/B)\to B$ arises in the same way as before.

The objective of the present paper is to provide a Hodge-theoretic proof of this result. In \cite{M}, the universal family of pairs $(V,S)$ with fixed $S$ is considered as a universal family of log-Calabi--Yau 3-folds, and the resulting Lagrangian fibration is a log-version of the so called Donagi--Markman integrable system
\cite{DM}, which is the relative intermediate Jacobian of the universal family of gauged Calabi-Yau 3-folds. The pairs $(Z,Y)$ are no more log-Calabi--Yau's, but have a summand in their mixed Hodge structure (MHS), similar
to the MHS of a log-Calabi--Yau, and this suffices to produce a symplectic structure making the relative intermediate Jacobian into a Lagrangian fibration.

Another example of this kind is described in \cite{IM3}, where $Y$ is a fixed 4-dimensional section of the Grassmannian $G(2,5)\subset \PP^9$ by a hyperplane and a hyperquadric, while $Z$ runs over the extensions of $Y$ to 5-dimensional sections of $G(2,5)$ by hyperquadrics. More cases when the Hodge structure or the derived category of a higher-dimensional variety contains K3-like or Calabi--Yau-like pieces are presented in \cite{KM}, \cite{KMM}, \cite{IM4}.

In Sect. \ref{MHS}, we compute the MHS of the pair $(Z,Y)$ and the Gauss--Manin connection on the variation of the MHS over the universal family $\ZZZ\to B$ of pairs $(Z,Y)$ with $Y$ fixed.
In Sect.~\ref{ss}, we verify (Lemma \ref{jb}) that the tangent space to $J(Z)$ at zero is naturally dual $T_ZB$, so that  $\JJJ(\ZZZ/B)$
is a quotient of the cotangent bundle $T^*B$ by a local system of lattices $\LLL\subset T^*B$ of rank 42. Next we show (Theorem \ref{alpha_descends}) that the local sections of $\LLL$ are Lagrangian w. r. t. the standard symplectic form $\tilde\alpha$
on $T^*B$, so that $\tilde\alpha$ descends to a symplectic form $\alpha$ on the quotient
$T^*B/\LLL=\JJJ(\ZZZ/B)$.

{\sc Acknowledgements.}  The author acknowledges the support of the French Agence Nationale de Recherche VHSMOD-2009 Nr. ANR-09-BLAN-0104.


\section{Mixed Hodge structures and their variations}
\label{MHS}

\begin{lemma}\label{coh_omega}
Let $Z$ be a nonsingular $5$-dimensional cubic in $\PP^6$.
Then the only nonzero numbers $h^i(Z,\Omega_Z^j(k))$ with $i>0, j>0, k\geq 0$ are the following ones:

(a) $h^i(\Omega_Z^i)=1$ for $i=1,\ldots,5$.

(b) $h^3(\Omega_Z^2)=21$, $h^3(\Omega_Z^2(1))=7$, $h^3(\Omega_Z^2(2))=1$.

(c) $h^2(\Omega_Z^3)=21$, $h^2(\Omega_Z^3(1))=h^2(\Omega_Z^3(2))=35$,
$h^2(\Omega_Z^3(3))=21$, $h^2(\Omega_Z^3(4))=7$, $h^2(\Omega_Z^3(5))=1$.

(d) $h^1(\Omega_Z^4(k))=C_7^{k-1}$ for $1\leq k\leq 8$.
\end{lemma}

\begin{proof}
Standard.
\end{proof}

Let $Y\subset Z$ be a smooth hyperplane section of $Z$, $U=Z\setminus Y$, $u:U\hookrightarrow Z$ and $i:Y\hookrightarrow Z$ natural inclusions. We will use the techniques
explained in Sect. 6.1 of \cite{V}, which are applied there to the construction
of the Griffiths residue map, but they also provide a description
of the mixed Hodge structure (MHS) on the cohomology of $U$
via meromorphic forms on $Z$ with poles along $Y$. See Ch. 4 of \cite{KK} or \cite{De}
for the MHS and their variations.

Consider the Gysin exact sequence
\begin{equation}\label{ces}
	0=H^3(Y,\ZZ) \xrightarrow{\ }  H^5(Z,\ZZ)\xrightarrow{u^*}  H^5(U,\ZZ)
	\xrightarrow{\Res}	H^4(Y,\ZZ)\xrightarrow{i_*} H^6(Z,\ZZ)\rar H^6(U,\ZZ)=0.
\end{equation}
It exhibits two levels of the weight filtration of the MHS on $H^5(U)$: 
$$
W_4H^5(U)=0\subset W_5H^5(U)=u^*(H^5(Z))\subset W_6H^5(U)=H^5(U).
$$
The two graded factors are $\Gr_5^W=W_5$, a pure Hodge structure of weight 5, 
isomorpic to that of $H^5(Z)$, and $\Gr_6^W=W_6/W_5
\simeq \ker[ H^4(Y)\xrightarrow{i_*} H^6(Z)](-1)$, a pure Hodge structure
of weight~6. Here $-1$ between parentheses stands for the Tate twist raising
degree by 2, so that
$\ker i_*$, which is a pure Hodge substructure of weight 4 in $H^4(Y)$, is
transformed into a pure Hodge structure of weight 6.

We will now describe the Hodge filtration
$$
F^0\supset F^1\supset \ldots \supset F^4\supset F^5
$$
on $H^5(U)$. To this end, we represent $H^5(U)$ as the hypercohomology
$\HH^5(\Omega^\fatdot_Z(\log Y))$. Then $F^p$ is first defined on the level
of complexes,
$$
F^p\Omega^\fatdot_Z(\log Y)=[\Omega^p_Z(\log Y)\rar
\Omega^{p+1}_Z(\log Y)\rar \ldots],
$$
and $F^pH^5(U)$ is the image of the natural map
$\HH^5(F^p\Omega^\fatdot_Z(\log Y))\rar \HH^5(\Omega^\fatdot_Z(\log Y))$.
By Deligne's theorem on the degeneration of the Fr\"olicher spectral sequence,
the latter map is injective. To relate the Hodge and the weight filtrations,
we recall that the latter can be defined on the logarithmic de Rham
complex by setting $W_k\Omega^\fatdot_Z(\log Y)$ to be the subcomplex
of logarithmic forms with at most $k$ irreducible polar divisors,
and then $W_{d+k}H^d(U)=\HH^d(W_k\Omega^\fatdot_Z(\log Y))$. In our case
$Y$ is the only irreducible divisor that may be polar, so there are only
two levels of the weight filtration: $W_0=\Omega^\fatdot_Z$
and $W_1=\Omega^\fatdot_Z(\log Y)$.

Let us introduce the sheaves $\Omega^p_Z(mY)$ of meromorphic $p$-forms with 
poles of order $\leq m$ along $Y$, and their subsheaves $\Omega^{p,c}_Z(mY)$
of closed meromorphic $p$-forms. Then $\Omega^{p,c}_Z(Y)\subset \Omega^p_Z(\log Y)$,
and $F^p\Omega^\fatdot_Z(\log Y)$ is quasi-isomorphic to the singleton
$\Omega^{p,c}_Z(Y)[-p]$, the latter being the complex with only one nonzero term
$\Omega^{p,c}_Z(Y)$ placed in degree $p$. Thus $F^pH^5(U)=H^5(\Omega^{p,c}_Z(Y))$.

\begin{lemma}\label{hodge_numbers}
For the Hodge filtration on $H^5(U)$, we have: $F^0=F^1=F^2=H^5(U)$, $\dim H^5(U)=64$,
$\dim F^3=42$, $\dim F^4=1$, $F^5=0$. The nonzero Hodge numbers $h_k^{p,q}$ of
the graded factors $\Gr_k^WH^5(U)$ are $h_5^{2,3}=h_5^{3,2}=21$ and
$h_6^{2,4}=h_6^{4,2}=1$, $h_6^{3,3}=21$.
\end{lemma}

\begin{proof}
This follows from Lemma \ref{coh_omega} and from the exact triples
$$
0 \xrightarrow{\ }\Omega^{r,c}_Z(mY)\xrightarrow{\ }\Omega^{r}_Z(mY)\xrightarrow{d}\Omega^{r+1,c}_Z((m+1)Y)\xrightarrow{\ }0
$$
for $r\geq 1$, $m\geq 1$. 
\end{proof}

In particular, the calculation of the Hodge space $F^4H^5(U)$
along the lines of the proof of Lemma \ref{hodge_numbers} goes as follows.
The exact triple
$$
0 \xrightarrow{\ }\Omega^{4,c}_Z(Y)\xrightarrow{\ }\Omega^{4}_Z(Y)\xrightarrow{d}\Omega^{5,c}_Z(2Y)\xrightarrow{\ }0
$$
together with the observations that $\Omega^{5,c}_Z(2Y)=\Omega^{5}_Z(2Y)
\simeq\OOO_Z(-2)$ and $h^i(\OOO_Z(-2))=0$ for all $i$ imply
the isomorphisms $H^i(\Omega^{4,c}_Z(Y))=H^i(\Omega^{4}_Z(Y))$.
From Lemma \ref{coh_omega}, we deduce that $F^4H^5(U)=H^1(\Omega^{4,c}_Z(Y))$
is 1-dimensional. 

Denote by $\omega_Z$ a generator of this 1-dimensional
space. It is convenient to consider it in the Dolbeault representation as
a $\bar\partial$-closed $(0,1)$-form with values in the holomorphic
vector bundle $\Omega^{4}_Z(Y)$. The exact triple
$$
0 \xrightarrow{\ }\Omega^{4}_Z(\log Y)\xrightarrow{\ }\Omega^{4}_Z(Y)\xrightarrow{\ }\Omega^{4}_Y\otimes \NNN_{Y/Z}\xrightarrow{\ }0
$$
and the vanishing of the cohomology of its third term
allow us to choose $\omega_Z$ with values in $\Omega^{4}_Z(\log Y)$.

We will now proceed to families of objects introduced above,
varying over a smooth base $B$. 
They will be denoted by the respective calligraphic letters: $(\ZZZ,\YYY),
\UUU=\ZZZ\setminus \YYY, \HHH^5(\UUU/B), \FFF^p=\FFF^p\HHH^5(\UUU/B), \WWW_k=\WWW_k\HHH^5(\UUU/B)$. The individual fibers over a point $b\in B$ will be denoted
by $(Z_b,Y_b), U_b, H^5(U_b)$, etc. Here $\WWW_k$ are local systems, defined over $\ZZ$,
and $\FFF^p$ are holomorphic vector bundles. By abuse of notation, we will denote by the same
symbol a local system, the associated flat $\CC$-vector bundle and the sheaf of its
holomorphic sections. The Gauss-Manin connection will be denoted by $\nabla$.

The first order infinitesimal deformations
of the pair $(Z,Y)=(Z_b,Y_b)$ are classified by $H^1(Z, \TTT_{Z}(-\log Y))$,
where $\TTT_{Z}(-\log Y)$ is the logarithmic tangent vector bundle. It is defined
via its sheaf of sections, which
is the subsheaf of $\TTT_{Z}$ consisting of all the vector fields that preserve the
ideal sheaf of $Y$ in $\OOO_{Z}$. For a point $b\in B$, the natural classifying map
$\rho: T_bB\rar H^1(Z, \TTT_{Z}(-\log Y))$ is called the Kodaira-Spencer map.

For any tangent vector $v\in T_bB$, the Gauss-Manin connection defines the covariant derivative
$\nabla_v\in \End(H^5(U_b))$, and the Griffiths transversality sais that
$\nabla_v(F^p)\subset F^{p-1}$. According to Prop. 4.4 from \cite{U}, the induced
map $\bar\nabla_v:\Gr^p_FH^5(U_b)\rar \Gr^{p-1}_FH^5(U_b)$ is the contraction with
the Kodaira-Spencer class $\rho(v)\in H^1(Z_b, \TTT_{Z_b}(-\log Y_b))$.

Let $\omega$ be a holomoprhic generating section of the line bundle $\FFF^4$,
and $\omega_b\in F^4H^5(U_b)$ its value at $b$. Let $v\in T_bB$ and $\rho(v)$
its Kodaira-Spencer class. Let us think of both $\omega_b, \rho(v)$ as
$\bar\partial$-closed $(0,1)$-form with values in $\Omega^{4}_Z(\log Y)$, $\TTT_{Z}(-\log Y)$
respectively. Then the result of contraction $\rho(v)\lrcorner\,\omega_b$ is
a $\bar\partial$-closed $(0,2)$-form with values in $\Omega^{3}_Z(\log Y)$,
whose cohomology class represents $\bar\nabla_v (\omega_b)$.

There is a particular case when we can assert that $\rho(v)\lrcorner\,\omega_b$
takes values in the regular 3-forms $\Omega^{3}_Z$. This is when the first order
deformation in the direction of $v$ fixes the first component of the pair $Y$.
According to Beauville \cite{B2}, the first order deformations of the pair
$(Y,Z)$ with fixed $Y$ are classified by $H^1(Z, \TTT_Z(-Y))$, so that
$\rho(v)$ takes value in vector fields vanishing along $Y$, and contracting it
with the 4-forms with a simple pole along $Y$, we obtain just regular 3-forms.

\section{Symplectic structure on the relative intermediate Jacobian}
\label{ss}

Let $Y$ be a fixed smooth cubic 4-fold, and $\pi:\ZZZ\rar B$ the universal family
over the moduli stack $B$ of nonsingular cubic 5-folds $Z=Z_b$ extending $Y$. According
to \cite{B2}, for any $b\in B$, $T_bB=H^1(Z, \TTT_Z(-Y))$, where the equality
sign stands for a canonical isomorphism provided by the Kodaira-Spencer map $\rho$.
Moreover,
the obstruction space for the deformation theory of pairs $(Z,Y)$ with fixed $Y$
is $H^2(Z, \TTT_Z(-Y))$.

\begin{lemma}
$B$ is a smooth Deligne--Mumford stack.
\end{lemma}

\begin{proof}
For smoothness, it is enough to verify that
$H^2(Z, \TTT_Z(-Y))=0$. We have $\TTT_Z(-Y)\simeq \Omega^4_Z(3)$, and the vanishing
of $H^2$ follows from Lemma \ref{coh_omega}. The automorphism groups
of the 5-folds $Z_b$ are finite, hence the moduli stack is a Deligne-Mumford one.
\end{proof}

For every fiber $Z_b$, we denote by $J(Z_b)$ the intermediate Jacobian
$H^{3,2}(Z_b, \CC)^*/H_5(Z_b,\ZZ)$ of $Z_b$, and
$\JJJ\rar B$ the relative intermediate Jacobian, $\JJJ=\HHH^{3,2}(\ZZZ/B)^*/\HHH_5(\ZZZ/B,\ZZ)$.
At any point $P\in J(Z_b)$, the fiber of the vertical tangent bundle
$\TTT_{\JJJ/B}|_P$ at $P$ is canonically identified with $H^{3,2}(Z_b, \CC)^*=T_PJ(Z_b)$.

\begin{lemma}\label{jb}
Let $b\in B$, $Z=Z_b$, and
$\omega_b$ be a generator of the $1$-dimensional vector space $H^1(\Omega^{4}_{Z}(Y))$.
Then the map $j_b:H^1({Z}, \TTT_{Z}(-Y))\rar H^2({Z}, \Omega^3_{Z})$ defined by contraction
with $\omega_b$ is an isomorphism.
\end{lemma}

\begin{proof}
Fixing a generator $\eta$ of the trivial vector bundle $\Omega^5_Z(4)$ provides us
with the isomorphisms $\TTT_Z(-3)\isom{\lrcorner\,\eta}\Omega^{4}_{Z}(Y)$
and $\wedge^2\TTT_Z(-4)\isom{\lrcorner\,\eta}\Omega^{3}_{Z}$, which transform
the contraction $\TTT_Z(-1)\times \Omega^{4}_{Z}(Y)\xrightarrow{\lrcorner} \Omega^{3}_{Z}$
into the wedge product
$\TTT_Z(-1)\times \TTT_Z(-3)\xrightarrow{\wedge} \wedge^2\TTT_Z(-4)$. We have to show that
the induced product map
\begin{equation}\label{prod}
H^1(Z,\TTT_Z(-3))\otimes H^1(Z,\TTT_Z(-1))\xrightarrow{\ } H^2(Z,\wedge^2\TTT_Z(-4))
\end{equation}
is an isomorphism.

Consider the normal bundle exact triple, twisted by $d-3$:
$$
0\lra \TTT_Z(d-3)\lra \TTT_{\PP^6}(d-3)|_Z\xrightarrow{\ \kappa_d} \OOO_Z(d)\lra 0.
$$
The surjection $\kappa_d$ is given by $x_i\frac{\partial}{\partial x_j}
\mapsto x_i\frac{\partial F}{\partial x_j}$, where $x_j$ are homogeneous
coordinates in $\PP^6$, and $F$ is the cubic form defining $Z$ in $\PP^6$.
Thus the connecting maps of the associated long exact sequence provide the isomorphisms
$\mu_d: R_{d}\isoto H^1(Z, \TTT_Z(d-3))$, where $R_d=H^0(Z, \OOO_Z(d))
/\im H^0(\kappa_d)$ is nothing else but the $d$-th homogeneous component
of the Jacobian ring $R=\CC [x_0,\ldots,x_6]/(\frac{\partial F}{\partial x_0},
\ldots,\frac{\partial F}{\partial x_6})$. 

Further, applying $\wedge^2$ to the normal bundle exact triple and twisting
by $-4$, we obtain the exact triple
$$
0\lra \wedge^2\TTT_Z(-4)\lra
\wedge^2\TTT_{\PP^6}(-4)|_Z\lra \TTT_Z(-1)\lra 0
$$
together with the connecting isomorphism
$\delta: H^1(Z,\TTT_Z(-1))\isoto H^2(Z,\wedge^2\TTT_Z(-4))$. Explicitizing the
connecting maps $\mu_0, \delta$, one can check the commutativity of the diagram
$$
\xymatrix{
H^0(Z,\OOO_Z)\otimes H^1(Z,\TTT_Z(-1)) \ar[rr] \ar[d]_{\mu_0\otimes\id} && 
H^1(Z,\TTT_Z(-1)) \ar[d]^{\delta} \\
H^1(Z,\TTT_Z(-3))\otimes H^1(Z,\TTT_Z(-1)) \ar[rr] && H^2(Z,\wedge^2\TTT_Z(-4)) 
}
$$
The upper horizontal arrow is obviously an isomorphism, hence so is the bottom one.
\end{proof}

We can interpret $j_b$ as an isomorphism $T^bB\isoto H^{3,2}(Z_b, \CC)$.
Choosing a holomorphic nonvanishing section $\omega$ of the line bundle
$R^1\pi_*(\Omega^4_{\ZZZ/B}(1))$, we can relativize the construction
of $j_b$ to get an isomorphism of vector bundles over $B$:
$$
j:\TTT_B\isoto \HHH^{3,2}(\ZZZ/B, \CC).
$$
We will normalize the choice of $\omega$ by the condition
that $\Res\omega_b$ is a fixed generator of the 1-dimensional
vector space $F^3H^4(Y)=H^1(Y,\Omega^3_Y)$. This provides a
global isomorphism $j$, which we fix from now on. 

Let now look at the total space of the vector bundle $\EEE=\HHH^{3,2}(\ZZZ/B, \CC)^*$.
The transpose $j'$ of $j$ identifies it with the cotangent bundle of $B$, so
$\EEE$ carries a natural holomorphic symplectic form, which we will denote by $\tilde\alp$.
The relative intermediate Jacobian $\JJJ=\JJJ(\ZZZ/B)$ is the quotient
$\EEE/\LLL$, where $\LLL=\HHH_5(\ZZZ/B,\ZZ)$ is the local system of lattices
$\ZZ^{42}$ in $\EEE$.

\begin{theorem}\label{alpha_descends}
The locally constant sections of the local system $\LLL$ over open
subsets of $B$ are Lagrangian with respect to $\tilde\alp$. This
implies that $\tilde\alp$ descends to a symplectic form $\alp$ on the quotient
$\JJJ=\EEE/\LLL$ such that the structure map $\pi:\JJJ\rar B$ is a Lagrangian
fibration.
\end{theorem}

\begin{proof}
The proof goes exactly as in \cite{M} for K3-Fano flags $(S,X)$ in place of
the flag $(Y,Z)$ of cubic hypersurfaces. We use the observation of Donagi--Markman
\cite{DM}, that a local section $\gamma$ of $\LLL$ is Lagrangian if an only if
the local section $j'(\gamma)$ of $\Omega^1_B$ is a closed 1-form. Explicitly,
the 1-form $j'(\gamma)$ is described as follows: its value on a tangent vector
$v\in T_bB$ is given by
$$
j'(\gamma)(v)=\int_{\gamma_b}v\lrcorner\,\omega_{b}.
$$

Here, as in Sect. \ref{MHS}, we use the Dolbeault representatives of the cohomology classes,
so that $v\lrcorner\,\omega_{b}$ is a $\bar\partial$-closed (3,2)-form on
$Z$, and the integral stands for the pairing between homology and
cohomology classes of complementary degrees. To see that $j'(\gamma)$ is closed, we
will represent it as $dg$ for some function $g$, defined locally on $B$.
Take any local lift $\tilde\gamma$ of $\gamma$ under the natural epimorphism
$\HHH_5(\UUU/B)\twoheadrightarrow \HHH_5(\ZZZ/B)$, and define the function $g$ by 
$g(b)=\int_{\tilde\gamma_b}\omega_{b}$. Its differential at $b$ is expressed
in terms of the Gauss-Manin connection:
$$
d_bg(v)=\int_{\tilde\gamma_b}\nabla_v\omega (b).
$$
As we saw in Sect. \ref{MHS}, $\nabla_v\omega (b)\equiv v\lrcorner\,\omega_{b}
\mod F^4H^5(U_b)$, so that there exists a constant $c$, possibly depending on $b$,
such that $\nabla_v\omega (b)=v\lrcorner\,\omega_{b}+ c\omega_{b}$, where
$v\lrcorner\,\omega_{b}\in W_5H^5(U_b)$. Now we remind that, first, the map $\Res$
annihilates $W_5$, second, it is defined
over $\ZZ$, hence is covariantly constant, and third, we have chosen $\omega$
in such a way that
$\Res (\omega_b)=\eta$ does not depend on $b$. This implies that
$$
0=\nabla_v \Res(\omega)=\Res (\nabla_v\omega)=\Res (v\lrcorner\,\omega_{b}+ c\omega_{b})
=c\eta.
$$
Hence $c=0$ and $j'(\gamma)=dg$. This ends the proof. 
\end{proof}

\end{document}